\documentclass{article}[12pt]
\usepackage{amsmath}
\usepackage[pdftex]{graphicx}
\usepackage{textcomp}
\usepackage{pgf,tikz}
\usetikzlibrary{arrows}
\usepackage{amsthm}
\usepackage{enumerate}
\usepackage{mathrsfs}
\usepackage{bbm}
\usepackage{color}
\usepackage{amssymb}

\newtheorem{proposition}{Proposition}
\newtheorem{theorem}{Theorem}
\newtheorem{lemma}{Lemma}
\newtheorem{remark}{Remark}
\newcommand{\p}{\mathbbm{P}}
\newcommand{\e}{\mathbbm{E}}
\newcommand{\N}{\mathbbm{N}}

\newcommand{\A}{\mathcal{A}}
\newcommand{\LL}{\mathscr{L}}
\newcommand{\HH}{\mathcal{H}}
\newcommand{\M}{\mathcal{M}}

\begin{document}

\title{Median inverse problem and approximating the number of $k$-median inverses of a permutation\thanks{Partially supported by CNPq, FAPERJ and NSERC. DS holds the Canada Research Chair in Mathematical Genomics.}}

\author{Poly H. da Silva$^1$, Arash Jamshidpey$^2$ and David Sankoff$^3$}

 \footnotetext[1]{Fluminense Federal University, Brazil}
 \footnotetext[2]{IMPA, Brazil}
 \footnotetext[3]{University of Ottawa, Canada}
\maketitle

\begin{abstract}
We introduce the ``\textit{Median Inverse Problem}'' for metric spaces. In particular, having a permutation $\pi$ in the symmetric group $S_n$ (endowed with the breakpoint distance), we study the set of all $k$-subsets $\{x_1,...,x_k\}\subset S_n$ for which $\pi$ is a breakpoint median. The set of all $k$-tuples $(x_1,...,x_k)$ with this property is called the $k$-median inverse of $\pi$. Finding an upper bound for the cardinality of this set, we provide an asymptotic upper bound for the probability that $\pi$ is a breakpoint median of $k$ permutations $\xi_1^{(n)},...,\xi_k^{(n)}$ chosen uniformly and independently at random from $S_n$.
\end{abstract}

\section{Introduction}\label{section1}

Each genome has a specific order of genes (markers or sites) in its chromosome. In the case that duplicated genes do not exist, we can represent each unichromosomal genome by a permutation. A simple way to compare two unichromosomal genomes with the same set of genes is to measure their dissimilarities, that is, if in the first genome two genes are adjacent and in the second one they are not adjacent, we have then a \emph{breakpoint}. The total number of breakpoints is a metric of dissimilarity, introduced by Sankoff and Blanchette~\cite{sankoff97} in 1997.

On the other hand, one can use the definition of \emph{median}~\cite{sankoff97} to compare three or more genomes. Given a set of genomes $X=\{g_1,...,g_k\}$ (all permutations with the same length) and a genomic distance $\rho$, a median for the set $X$ is a genome that minimizes the total distance function $\rho_T(\cdot,X):=\sum_{i=1}^k \rho(\cdot,g_i)$. The minimum value of $\rho_T(.,X)$ is called the median value of the set $X$. 
The median problem is the problem of finding a median for a set of given genomes $X$.

Medians play an important role in small and large phylogeny problems. In some evolutionary models, at least one of the medians of some species carries valuable information about their first common ancestor or even about the phylogenetic tree. The median problem has been extensively studied for different genome distances, and for most of them, including breakpoint distance, it is an NP-hard problem~\cite{bryant98,caprara03,tannier09,fertincombinatorics}.


 In this paper, we introduce the \textit{``Median Inverse Problem'' (MIP)}, as an alternative approach to study the breakpoint medians, that is, for a given permutation $\pi$, we study the probability that $\pi$ is a median of a set of $k$ permutations chosen uniformly and independently at random from 
$S_n$. This opens another way to study medians of random sets. More formally, denote by $\M_n(B)$ the set of all breakpoint medians of $B\subset S_n$. Given a set $A\subset S_n$, we look for all $k$-tuples $(x_1,...,x_k)\in S_n^k$ with $A\subset \M_n(\{x_1,...,x_k\})$. That is, the $k$-median inverse of $A$ is defined by
\[
\M_{n,k}^{-1}(A):=\{(x_1,...,x_k)\in S_n^k: \ A\subset \M_n(\{x_1,...,x_k\})\}.
\]
We find an upper bound for the cardinality of the above set when $A$ is singleton. We approximate the asymptotic probability that $k$ independent random permutations in $S_n$ have a given median $\pi\in S_n$, and find an upper bound for it. More precisely, an adjacency of a permutation $x=x_1 \ ...\ x_n$ is an unordered pair of adjacent numbers $\{x_i,x_{i+1}\}=\{x_{i+1},x_i\}$, for $i=1,...,n-1$. Denote by $\mathcal{A}_x$, the set of all adjacencies of permutation $x$. In general, for $k$ permutations $X=\{x_1,...,x_k\}$, there may exist a median $\pi$ such that $\mathcal{A}_\pi\setminus \cup_{i=1}^k \mathcal{A}_{x_i} \neq \emptyset$. We find an upper bound for $\max\{|\mathcal{A}_\pi\setminus \cup_{i=1}^k \mathcal{A}_{x_i}|: \ \pi\in \M_n(X)\}$. Denote this upper bound by $\mathcal{O}_n(X)$. Let $\xi_1^{(n)},...,\xi_k^{(n)}$ be $k$ permutations chosen uniformly and independently at random from $S_n$. Then, for any sequence of real numbers $(a_n)_{n\in \N}$ tending to $\infty$, we show that $|\mathcal{O}_n(\{\xi_1^{(n)},...,\xi_k^{(n)}\})|/a_n\rightarrow 0$, in probability, as $n$ goes to $\infty$. For $c\geq 0$, letting
\[
\mathscr{L}_{n,k,c}^{-1}(\{\pi\})=\{(x_1,...,x_k)\in S_n^k: |\mathcal{A}_\pi\setminus \bigcup\limits_{x\in X} \mathcal{A}_x|\leq c\},
\]
we prove that
\[
\limsup_n \mathbb{P}((\xi_1^{(n)},...,\xi_k^{(n)})\in M_{n,k}^{-1}(\{\pi_n\})) \leq \limsup_n \mathbb{P}((\xi_1^{(n)},...,\xi_k^{(n)})\in \mathscr{L}_{n,k,a_n}^{-1}(\{\pi_n\})),
\]
for any arbitrary sequence of $(\pi_n)_{n\in \N}$, with $\pi_n\in S_n$. Finally we give an exact expression for the number of elements of $\mathscr{L}_{n,k,a_n}^{-1}(\{\pi_n\})$.

\section{Background}
A permutation of length $n$ is a bijection on $[n]:=\{1,...,n\}$. A permutation $\pi$ is denoted by
\[
{1 \ \ 2 \ \ ... \ \ \ n \choose \pi_1 \ \pi_2 \ \ ... \ \ \pi_n},
\]
or simply by $\pi_{1} \ \pi_2 \ \ ... \ \ \pi_{n}$. We represent a linear unichromosomal genome with $n$ genes or markers by a permutation of length $n$. Each number represents a gene or a marker in the genome. The set of all permutations of length $n$ with the function composition operator is a group called \textit{symmetric group} of order $n$ denoted by $S_n$. We denote by $id:=id^{(n)}$ the identity permutation $1 \ 2 \ 3 \ ... \ n$. For a permutation $\pi := \pi_1 \  ... \ \pi_n$, any unordered pair $\{\pi_i , \pi_{i+1}\}=\{\pi_{i+1} , \pi_i\}$, for $i=1, ..., n-1$, is called an adjacency of $\pi$. We denote by $\mathcal{A}_{\pi}$ the set of all adjacencies of $\pi$ and by $\mathcal A_{x_1,...,x_k}$ the set of all common adjacencies of $x_1,...,x_k\in S_n$. For any $x,y \in S_n$, the \textit{breakpoint distance} (bp distance) between $x$ and $y$ is defined by $d^{(n)}(x,y):=n-1-|\mathcal A_{x,y}|$ which is a pseudometric. 
We say a pseudometric (or a metric) $\rho$ is left-invariant on a group $G$ if for any $x,y,z \in G$, $\rho(x,y)=\rho(zx, zy)$. The bp distance is a left-invariant pseudometric on $S_n$. We say two permutations $\pi$ and $\pi'$ in $S_n$ are equivalent, denoted by $\pi\sim \pi'$, if $d^{(n)}(\pi,\pi')=0$. In other words they are equivalent if $\pi_i=\pi'_{n+1-i}$, for $i=1,...,n$. The set of all equivalence classes of $S_n$ under $\sim$, denoted by $\hat{S}_n:=S_n/\sim$, endowed with $d^{(n)}$ is a metric space.

For a metric (or pseudometric) space $(S,\rho)$, $z\in S$ is said to be a geodesic point of $x,y\in S$, if $\rho(x,y)=\rho(x,z)+\rho(z,y)$. In the case of permutations (or permutation classes), we also call a geodesic point, a geodesic permutation (or a geodesic permutation class). 
We denote by $\overline {[x,y]}_S$ the set of all geodesic points of $x,y$ in a metric space $(S,\rho)$, and, in particular, by $\overline {[x,y]}$ the set of all geodesic points of $x,y\in \hat{S}_n$, with the breakpoint distance.

For a metric (or pseudometric) space $(S,\rho)$, let us define the total distance of a point $x\in S$ to a finite subset $A\subset S$ by
\[
\rho_T(x,A):=\sum\limits_{y\in A} \rho(x,y).
\]
A \emph{median} of a finite subset $A\subseteq S$  is a point of $S$ (not necessarily unique) whose total distance to $A$ takes the infimum (respectively, minimum for a finite space $S$), i.e.  a point $x\in S$ such that
\[
\rho_T(x,A)=\inf_{y\in S} \rho_T(y,A).
\]
For the finite space $S$, ``\textit{inf}'' is replaced by ``\textit{min}'' in the above definition, that is $x\in S$ is a median of $A$ if it minimizes the total distance function $\rho_T(.,A)$. Furthermore, the median value of $A$, denoted by $\mu(A)$, is the infimum (respectively, minimum) value of the total distance function to $A$. We denote by $\mathcal{M}_{S,\rho}(A)$ the set of all medians of $A$. In particular, we denote by $d_T^{(n)}(x,A)$ the total breakpoint distance of permutation $x\in S_n$ to $A\subset S_n$, and by $\mathcal{M}_n(A)$ the set of all breakpoint medians of $A$, that is $\mathcal{M}_n(A):=\mathcal{M}_{S_n,d^{(n)}}(A) $. There always exists a median (not necessarily unique) for any subset of a finite metric space, while this is not true for general infinite metric spaces. In the simple case of two points $x$ and $y$ in a general metric space, it is clear from definition that every median of $x$ and $y$  is a geodesic point of them and vice versa. That is, $\overline {[x,y]}$ is the set of medians of $x$ and $y$.

Common adjacencies of permutations can be regarded as a set of segments. A segment (of $S_n$) is a set of consecutive adjacencies of a permutation of length $n$. More explicitly, a segment of length $k\in [n-1]$ is a set of adjacencies 
\[
\{\{n_0,n_1\},\{n_1,n_2\},...,\{n_{k-2},n_{k-1}\},\{n_{k-1},n_k\}\},
\]
where $n_0,n_1,...,n_k\in [n]$ are different natural numbers. It also can be denoted by $[n_0,n_1,...,n_k]$ or equivalently by $[n_k,...,n_1,n_0]$. In particular, any segment of length $n-1$ is the set of adjacencies of a permutation (class) and vice versa. By convention, we assume that the empty set $\emptyset$ is a segment. We say a segment $s$ is a subsegment of a segment $s'$ if $s\subset s'$. For a given permutation $\pi=\pi_1 \ ... \ \pi_n\in S_n$, for $i\leq j$, the segment $[\pi_i,\pi_{i+1},...,\pi_j]=[\pi_j,...,\pi_{i+1},\pi_i]$ is denoted by $s_{ij}=s_{ij}^\pi$ and is called a segment of $\pi$. We denote by $|s|$ the length of a segment $s$. 
Note that a segment is originally defined as a set of adjacencies and therefore all set operations can be applied on it. Two segments $s=[n_0,...,n_k]$ and $s'=[m_0,...,m_{k'}]$ are said to be strongly disjoint if $\{n_0,...,n_k\}\cap \{m_0,...,m_{k'}\}=\emptyset$. They are disjoint if $s\cap s'=\emptyset$, otherwise we say that they intersect. 

Also, by a set of segments (segment set) of $S_n$, we mean the union of some pairwisely strongly disjoint segments of $S_n$. In other words, a set of segments or a segment set $I$ is a subset of $\mathcal{A}_\pi$ for a permutation $\pi$. In this case, we say $I$ is a segment set of $\pi$ or $\pi$ contains $I$. It is clear that a segment set can be contained in more than one permutation, or in other words, it can be contained in the intersection of adjacencies of several permutations. Denote by $\mathcal{I}^{(n)}$ the set of all segment sets of $S_n$. By a segment (or component) of a segment set $I$ we mean a maximal segment contained in $I$, and to show a segment $s$ is a segment of $I$, we denote $s\hat{\in} I$. Although a segment set $I$ containing segments $s_1,...,s_k$ is in principle the union of adjacencies of $s_i$'s, that is $I=\cup_{i=1}^k s_i$, to ease the notation, we sometime denote it by $\{s_1,...,s_k\}$. Also we denote by $\|I\|:=k$, the number of segments of $I$. Note that the notation $|\ .\ |$ is used for both cardinality of a set and absolute value of a real number. For example, as we already indicated for a segment $s$, $|s|$ is the number of adjacencies of $s$, and,  by the original definition of a segment set as a union of segments, 
\[
|I|=\sum\limits_{i=1}^{\|I\|} |s_i|
\]
is the number of adjacencies of $I$. 

It is clear that the intersection of segments (and in general, the intersection of segment sets) is  always a segment set. Two segment sets $I$ and $J$ (in particular two segments $s$ and $s'$, respectively) are said to be consistent, if their union is contained in $\mathcal{A}_\pi$, for a permutation class $\pi$. In particular, any two segment sets of a permutation $\pi$ are consistent. For example, for $n=10$, two segments $[3,7,10,2,5]$ and $[2,5,8,1]$ are consistent and their union is the segment $[3,7,10,2,5,8,1]$, while two segment $[2,6,3,8,1]$ and $[8,1,4,7,6,3,5]$ are not consistent. When we speak of union of two or more segment sets (respectively, two or more segments) we always assume that they are, pairwisely, consistent. 
We say segment sets $I_1,...,I_k$ completes each others if there exists a permutation $\pi$ such that $\cup_{i=1}^k I_i=\mathcal{A}_\pi$.

\section{Results}
We investigate the probability that $k$ permutations chosen uniformly and independently at random from $S_n$ has a specified median $\pi\in S_n$. To this end, we introduce the``\textit{Median Inverse Problem (MIP)}'' as follows. Given a set of permutations $A\subset S_n$, find the set of all $k$-tuples in $S_n^k$ for which any element of $A$ is a median. The set of all $k$-tuples with the mentioned property is called the $k$-median inverse of $A$. This actually opens another way to study medians. Recall that for a finite metric (or pseudometric) space $(S, \rho)$, we denote by $\mathcal{M}_{S,\rho}(A)$ the set of all medians of $A$, and in particular, $\mathcal{M}_n(A):=\mathcal{M}_{S_n,d^{(n)}}(A)$ is the set of all breakpoint medians of $A\subset S_n$. Given a point $x\in S$, we look for all $k$-tuples $(x_1,...,x_k)\in S^k$ with $x\in \mathcal{M}_{S,\rho}(\{x_1,...,x_k\})$, that is the $k$-median inverse of $x$ is
\[
\mathcal{M}_{S,\rho,k}^{-1}(x):=\{(x_1,...,x_k)\in S^k: \ x\in \mathcal{M}_{S,\rho}(\{x_1,...,x_k\})\}.
\]
For the special case of breakpoint $k$-median inverse of a permutation $\pi\in S_n$, we write
\[
\mathcal{M}_{n,k}^{-1}(\pi):=\{(x_1,...,x_k)\in S_n^k: \ x\in \mathcal{M}_n(\{x_1,...,x_k\})\}.
\]
For example, if $(S,\rho)$ is a pseudometric space,
\[
\mathcal{M}_{S,\rho,1}^{-1}(x)=[x]:=\{y\in S: \rho(x,y)=0\}
\]
and if it is a metric space then 
\[
\mathcal{M}_{S,\rho,2}^{-1}(x)=\{(x_1,x_2)\in S^2: \ x\in\overline{[x_1,x_2]}\}.
\]
The notion of median inverse can be easily generalized as follows. Let $A\subset S$. The $k$-median inverse of $A$ is defined by
\[
\mathcal{M}_{S,\rho,k}^{-1}(A):=\{(x_1,...,x_k)\in S^k: \ A\subset \mathcal{M}_{S,\rho}(\{x_1,...,x_k\})\}.
\]
In particular, the breakpoint $k$-median inverse of $A \subset S_n$ is
\[
\mathcal{M}_{n,k}^{-1}(A):=\{(x_1,...,x_k)\in S_n^k: \ A\subset \mathcal{M}_n(\{x_1,...,x_k\})\}.
\]
In this paper we approximate the asymptotic probability that $k$ independent random permutations in $S_n$ have a given median $\pi\in S_n$ and find an upper bound for it.

For $k$ permutations with pairwise maximum distance $n-1$ from each other, one can see that a permutation $\pi$ is their median if and only if every adjacency of $\pi$ is an adjacency of at least one of those $k$ permutations. In other words, we have the following result.
\begin{proposition}[Jamshidpey et. al. \cite{jam14}]\label{extreme}
Let $X \subset S_n$ such that $d(x,y)=n-1$ for any $x,y\in X$. Then $\pi$ is a median of $X$ if and only if $\mathcal A_{\pi} \subset \bigcup\limits_{x\in X} \mathcal A_x$.
\end{proposition}
This is not true in general. For example, if $n=9$, $x=2 \ 7 \ 5 \ 6 \ 8 \ 3 \ 9 \ 4 \ 1$ and $\pi=6 \ 8 \ 9 \ 3 \ 4 \ 1 \ 2 \ 7 \ 5$, then every adjacency of $\pi$ is either an adjacency of $id=id^{(n)}$ or an adjacency of $x$, 
but $d^{(9)}(id,x)=7<d^{(9)}(id,\pi)+d^{(9)}(\pi,x)=8$. In fact, this happens because all common adjacencies of $id$ and $x$ must be adjacencies of $\pi$ in order to have  $\pi \in \overline{[id,x]}$ as stated in the next proposition.

\begin{proposition}[Jamshidpey \textit{et al.}\cite{jam14}]\label{dist}
Let $x,y \in S_n$. Then $z \in \overline{[x,y]}$ if and only if $\mathcal A_{x,y} \subset \mathcal A_z \subset \mathcal A_x \cup \mathcal A_y$.
\end{proposition}

Motivated by Proposition \ref{extreme}, for $X\subset S_n$, we let
\[
\LL_{n,0}(X):=\{\pi\in S_n: \mathcal{A}_\pi\subset \bigcup\limits_{x\in X} \mathcal{A}_x\},
\]
and, similarly to the notion of median inverse, we define
\begin{multline}
\LL_{n,k,0}^{-1}(\pi)=\{(x_1,...,x_k)\in S_n^k: \pi \in \LL_{n,0}(\{x_1,...,x_k\})\}=\\
\{(x_1,...,x_k)\in S_n^k: \mathcal{A}_\pi\subset \bigcup\limits_{x\in X} \mathcal{A}_x\}.
\end{multline}
In fact, Proposition \ref{extreme} implies that for $X\subset S_n$ with $d^{(n)}(x,y)=n-1$, for $x,y\in X$, we have $\mathcal{M}_n(X)=\LL_{n,0}(X)$. Letting
\[
V_{n,k}=\{(x_1,...,x_k)\in S_n^k: \ d^{(n)}(x_i,x_j)=n-1, for \ i\neq j\},
\]
we get, for any $\pi\in S_n$,
\[
\LL_{n,k,0}^{-1}(\pi)\cap V_{n,k}=\mathcal{M}_{n,k}^{-1}(\pi)\cap V_{n,k}.
\]
In other words, when we restrict ourselves to permutations on $V_{n,k}$, elements of $\mathcal{M}_{n,k}^{-1}(\pi)$ are identical to those of $\LL_{n,k,0}^{-1}(\pi)$. Let us continue this paper with computing $|\LL_{n,k,0}^{-1}(\pi)|$. To this end, we try to find an ordered $k$-tuple $(J_1,...,J_k)$ where every $J_i$, for $i=1,...,k$, is a segment set of $S_n$ such that
\[
\bigcup\limits_{i=1}^k J_i=\mathcal{A}_\pi,
\]
and then find all possible $k$-tuple $(x_1,...,x_k)\in S_n^k$ such that, for $i=1,...,k$, permutation $x_i$ contains exactly segment set $J_i$ from $\pi$, not anything more. We will see that this exactly gives us a way to count $|\LL_{n,k,0}^{-1}(\pi)|$. More precisely, for a segment set $I$ in a permutation $\pi$, we define $\HH_{\pi}^{(n)}(I)$ to be the set of all permutations having exactly segment set $I$ from $\pi$, that is
\[
\HH_{\pi}^{(n)}(I)=\{x\in S_n: \ \mathcal{A}_{\pi,x}=I\}.
\] 
Observe that, for permutations $x,y\in S_n$, $x\in \HH_{y}^{(n)}(I)$ if and only if $y\in \HH_{x}^{(n)}(I)$. Also, one can see that, for two non-identical segment sets of $\pi\in S_n$, namely $I\neq I'$, we have
\[
\HH_{\pi}^{(n)}(I)\cap \HH_{\pi}^{(n)}(I')=\emptyset,
\]
since if $x\in \HH_{\pi}^{(n)}(I)$ and $y\in \HH_{\pi}^{(n)}(I')$, then $\mathcal{A}_{\pi,x}=I\neq I'=\mathcal{A}_{\pi,y}$.

Let
\begin{equation}\label{index-partition}
\mathscr{P}_{k,0}^{(n)}(\pi)=\{(J_1,...,J_k)\in (\mathcal{I}^{(n)})^k: \ \bigcup\limits_{i=1}^k J_i=\mathcal{A}_\pi\}.
\end{equation}
If $\mathcal{J}=(J_1,...,J_k),\mathcal{J}'=(J_1',...,J_k')\in \mathscr{P}_{k,0}^{(n)}(\pi)$, such that $\mathcal{J}\neq \mathcal{J}'$, then
\[
(\HH_{\pi}^{(n)}(J_1)\times ...\times \HH_{\pi}^{(n)}(J_k)) \cap (\HH_{\pi}^{(n)}(J_1')\times ...\times \HH_{\pi}^{(n))}(J_k'))=\emptyset.
\]
Now, if $(x_1,...,x_k)\in \mathscr{L}_{n,k,0}^{-1}(\pi)$, then $\mathcal{A}_\pi \subset \cup_{i=1}^k \mathcal{A}_{x_i}$. Therefore, there exists $(J_1,...,J_k)\in \mathscr{P}_{k,0}^{(n)}(\pi)$ such that, for any $i=1,...,k$, $\mathcal{A}_{\pi,x_i}=J_i$ implying that $(x_1,...,x_k)\in \HH_{\pi}^{(n)}(J_1)\times ...\times \HH_{\pi}^{(n)}(J_k)$. On the other hand, if
\[
(x_1,...,x_k)\in \bigcup\limits_{(\tilde{J}_1,...,\tilde{J}_k)\in \mathscr{P}_{k,0}^{(n)}(\pi)} \HH_{\pi}^{(n)}(\tilde J_1)\times ...\times \HH_{\pi}^{(n)}(\tilde J_k),
\]
then there exists $(J_1,...,J_k)\in \mathscr{P}_{k,0}^{(n)}(\pi)$ such that $x_i\in \HH_{\pi}^{(n)}(J_i)$, for $i=1,...,k$, which means by itself $\mathcal{A}_{\pi,x_i}=J_i$. Thus
\[
\mathcal{A}_\pi=\bigcup\limits_{i=1}^k J_i=\bigcup\limits_{i=1}^k \mathcal{A}_{\pi,x_i} \subset \bigcup\limits_{i=1}^k \mathcal{A}_{x_i}.
\]
Hence,
\[
(x_1,...,x_k)\in \mathscr{L}_{n,k,0}^{-1}(\pi).
\]
In fact, we have proved the following proposition.
\begin{proposition}
Let $n,k$ be natural numbers, and $\pi$ be a permutation in $S_n$. Then
\[
\mathscr{L}_{n,k,0}^{-1}(\pi)=\bigcup\limits_{(\tilde{J}_1,...,\tilde{J}_k)\in \mathscr{P}_{k,0}^{(n)}(\pi)} \HH_{\pi}^{(n)}(\tilde J_1)\times ...\times \HH_{\pi}^{(n)}(\tilde J_k),
\]
and
\[
|\mathscr{L}_{n,k,0}^{-1}(\pi)|=\sum\limits_{(\tilde{J}_1,...,\tilde{J}_k)\in \mathscr{P}_{k,0}^{(n)}(\pi)} \prod\limits_{i=1}^k |\HH_{\pi}^{(n)}(\tilde J_i)|.
\]
\end{proposition}
So, knowing the number of elements of $\HH_{\pi}^{(n)}(\tilde J_i)$ has an important role in counting the number of elements of $\mathscr{L}_{n,k,0}^{-1}(\pi)$.

What can we say when the permutations are chosen uniformly and independently at random, from $S_n$? One can see that the situation in this case is similar to that in the case of permutations with maximum distances from each others. The following classic result can shed light into this. 
\begin{proposition}
Let $\xi^{(n)}$ be a permutation chosen uniformly at random from $S_n$, and let $(a_n)_{n\in \N}$ be a sequence of real number tending to $\infty$
. Then 
\begin{itemize}
\item[i)] $\e[d^{(n)}(id^{(n)},\xi^{(n)})]=n-1 -\frac{2(n-1)}{n}, \ n\in \N$.
\item[ii)] $var(d^{(n)}(id^{(n)},\xi^{(n)}))= (2-\frac{2}{n})(-1+\frac{2}{n})+\frac{4(n-2)^2}{n(n-1)},  \ n\in \N$.
\item[iii)] For any $\varepsilon >0$,
\[
\p(n-1-d^{(n)}(id^{(n)},\xi^{(n)})\geq \varepsilon a_n)\rightarrow 0,
\]
as $n\rightarrow \infty$.
\end{itemize}
\end{proposition}
\begin{proof}
For simplicity, we drop the superscript of $id^{(n)}$ in the following computations. For $i=1,...,n-1$, let $\chi_i=1$ if the $i$-th adjacency of $\xi^{(n)}$, namely $\{\xi_i^{(n)},\xi_{i+1}^{(n)}\}$ is an adjacency of $id$, and let $\chi_i=0$, otherwise. For $i=1,...,n-1$, we have
\[
\e[\chi_i]=\frac{2(n-1)}{n(n-1)}=\frac{2}{n}.
\]
We can write
\[
\begin{array}{l}
\e[|\mathcal{A}_{id,\xi^{(n)}}|]=\sum\limits_{i=1}^{n-1} \e[\chi_i]=\frac{2(n-1)}{n}.
\end{array}
\]
Also,
\[
\begin{array}{l}
var(|\mathcal{A}_{id,\xi^{(n)}}|)=\sum\limits_{i=1}^{n-1} \e[\chi_i^2]+2\sum\limits_{i<j} \e[\chi_i\chi_j]-\e[|\mathcal{A}_{id,\xi^{(n)}}|]^2\\
=\e[|\mathcal{A}_{id,\xi^{(n)}}|](1-\e[|\mathcal{A}_{id,\xi^{(n)}}|])+2\sum\limits_{j-i>1} \e[\chi_i\chi_j]+2\sum\limits_{j-i=1} \e[\chi_i\chi_j].
\end{array}
\]
But, for $j-i>1$, we consider two cases. First, the $i$-th adjacency of $\xi^{(n)}$ can be one of the $n-3$ adjacencies of $id$, namely $\{u,u+1\}$ for $u=2,...,n-2$, each with two different directions, i.e. either $\xi_i^{(n)}=u$ and $\xi_{i+1}^{(n)}=u+1$, or  $\xi_{i+1}^{(n)}=u$ and $\xi_i^{(n)}=u+1$. In this case there are $n-4$ adjacencies of $id$ (exclude $\{u,u+1\}$ and both of its neighbouring adjacencies) that $j$-th adjacency of $\xi_i^{(n)}$ can be identical to, each with two directions. The second case, is when the $i$-th adjacency of $\xi^{(n)}$ is either $\{1,2\}$ or $\{n-1,n\}$, each with two possible directions, and in this case there are $n-3$ adjacencies of $id$ (exclude the one chosen for $i$-th adjacency of $\xi^{(n)}$ and its unique neighbouring adjacency in $id$), that can be picked for the $j$-th adjacency of $\xi^{(n)}$, again each with two directions. In summary, for $j-i>1$,
\begin{multline}
\e[\chi_i\chi_j]=\p(\chi_i=1,\chi_j=1)=\\
\frac{(2(n-3)\times 2(n-4))+((2\times 2) \times (2(n-3))}{n(n-1)(n-2)(n-3)}=\frac{4}{n(n-1)}.
\end{multline}
But the number of ways one can choose $i,j$ such that $j-i>1$ is $1+...+(n-3)=(n-2)(n-3)/2$. So
\[
2\sum\limits_{j-i>1} \e[\chi_i\chi_j]=\frac{4(n-2)(n-3)}{n(n-1)}.
\]
Similarly, for $j-i=1$, the $i$-th and $i+1$-st adjacencies of $\xi^{(n)}$ should be identical to two consecutive adjacencies of $id$. Considering the direction, this gives $2(n-2)$ possible ways in which $\chi_i=1$ and $\chi_{i+1}=1$. This implies that, for $j-i=1$,
\[
\e[\chi_i\chi_j]=\p(\chi_i=1,\chi_j=1)=\frac{2(n-2)}{n(n-1)(n-2)}=\frac{2}{n(n-1)}
\]
 As the number of ways one can choose $i,j$ such that $j-i=1$ is $n-2$, we can write
\[
2\sum\limits_{j-i=1} \e[\chi_i\chi_j]=\frac{4(n-2)}{n(n-1)}.
\]
Having all together, we get
\[
var(|\mathcal{A}_{id,\xi^{(n)}}|)= \frac{2(n-1)}{n} (1-\frac{2(n-1)}{n})+\frac{4(n-2)(n-3)}{n(n-1)}+\frac{4(n-2)}{n(n-1)},  \ n\in \N
\]
which concludes the second part of the proposition. Finally, letting $n\rightarrow \infty$, we have $\e[|\mathcal{A}_{id,\xi^{(n)}}|]\rightarrow 2$ and $var(|\mathcal{A}_{id,\xi^{(n)}}|)\rightarrow 2$. Therefore, for any arbitrary sequence $(a_n)_{n\in \N}$, satisfying the conditions mentioned in the statement of the proposition, Chebyshev's inequality implies convergence in probability of $|\mathcal{A}_{id,\xi^{(n)}}|/{a_n}$ to $0$, as $n\rightarrow \infty$, and therefore, $(iii)$ is proved.
\end{proof}
As we mentioned before, when the pairwise distances of permutations in $A\subset S_n$ takes its maximum value $n-1$, a permutation $\pi$ is a median of $A$ if and only if each of its adjacencies is an adjacency of exactly one of the permutations in $A$. This is not true in general. In fact, for a general $A\subset S_n$, a median need not to take all of its adjacencies from $\cup_{x\in A} \mathcal{A}_x$. Can we find an upper bound for the number of adjacencies of any median of $A$ which are not from $\cup_{x\in A} \mathcal{A}_x$? In other words, can we find a good uniform upper bound for $|\mathcal{A}_\pi\setminus \cup_{x\in A} \mathcal{A}_x|$ , for any median $\pi$ of $A$? The next theorem answers this question. Before stating this theorem we introduce some notations as follows.

Denote by $\mathcal{P}(S)$ the set of all subsets of a set or space $S$. Let $X=\{x_1, ...,x_k\}\subset S_n$ and let $\mathcal{B}_X^X=\mathcal{B}_{x_1,...,x_k}^X:=\mathcal{A}_{x_1,...,x_k}$. Then, for any $j=1...k$, let
\[
\mathcal{B}_{x_1,...,x_{j-1},x_{j+1},...x_k}^X:=\mathcal{A}_{x_1,...,x_{j-1},x_{j+1},...x_k} \setminus \mathcal{B}_{x_1,...,x_k}^X
\]
Continuing this, for any $U=\{x_{i_1},...,x_{i_r}\}\subset X$, we set
\[
\mathcal{B}_U^X=\mathcal{B}_{x_{i_1},...,x_{i_r}}^X:=\mathcal{A}_U \setminus (\bigcup\limits_{U \subsetneqq V} \mathcal{B}_V^X).
\]
Also, for a permutation $\pi$ and $r\leq k$, let
\[
\bar{\varepsilon}_{i_1,...,i_r}^X(\pi):=|\mathcal{A}_\pi\cap \mathcal{B}_{x_{i_1},...x_{i_r}}^X|.
\]
For $x\in S_n$ and a subset $X\subset S_n$, the bp total distance of $x$ to $X$ is denoted by
\[
d_T(x,X)=d_T^{(n)}(x,X):=\sum\limits_{y\in X} d(x,y).
\]
\begin{theorem}\label{thmalpha}
Let $X=\{x_1,...,x_k\}\subset S_n$, and let $\pi\in \mathcal{M}_n(X)$. We assume the labels of elements of $X$ are such that 
\[
d_T^{(n)}(x_k,X)=\min\limits_{i=1...k} d_T^{(n)}(x_i,X).
\]
Then
\begin{multline}
|\mathcal{A}_x \setminus (\bigcup\limits_{i=1}^k \mathcal{A}_{x_i})|\\
\leq \sum\limits_{r=2}^{k} (r-1) \{\sum\limits_{1\leq i_1<...<i_r\leq k} \bar{\varepsilon}_{i_1,...,i_r}^X(\pi)- \sum\limits_{1\leq i_1<...<i_{r-1}< k} |\mathcal{B}_{i_1,...,i_{r-1},k}^X|\}\\
\leq \sum\limits_{r=2}^{k-1} (r-1)\sum\limits_{1\leq i_1<...<i_r<k} |\mathcal{B}_{x_{i_1},...,x_{i_r}}^X|.
\end{multline}
In particular, for $k=3$, for any $x\in \mathcal{M}_n(X)$
\[
|\mathcal{A}_x\setminus \bigcup\limits_{i=1}^3 \mathcal{A}_{x_i}| \leq |\mathcal{B}_{x_1,x_2}^X|.
\]
\end{theorem}
\begin{proof}
Let $\eta=|\mathcal{A}_\pi \setminus \cup_{x\in X} \mathcal{A}_x|$. Then
\[
\eta+\sum\limits_{r=1}^k \sum\limits_{1\leq i_1<...<i_r\leq k} \bar{\varepsilon}_{i_1,...,i_r}^X(\pi)=n-1
\]
As $\pi$ is a median of $X$, we have
\[
\begin{array}{l}
d_T^{(n)}(\pi,X)=k(n-1)- \sum\limits_{r=1}^k [r \sum\limits_{1\leq i_1<...<i_r\leq k} \bar{\varepsilon}_{i_1,...,i_r}^X(\pi)]\\
=(k-1)(n-1)+\eta-\sum\limits_{r=2}^k [(r-1) \sum\limits_{1\leq i_1<...<i_r\leq k} \bar{\varepsilon}_{i_1,...,i_r}^X(\pi)]\\
\leq d_T^{(n)}(x_k,X)=(k-1)(n-1)-(\sum\limits_{1\leq i_1< k} |\mathcal{B}_{i_1,k}^X|+2\sum\limits_{1\leq i_1<i_2< k} |\mathcal{B}_{i_1,i_2,k}^X|+...+\\
(k-2)\sum\limits_{1\leq i_1<...<i_{k-2}< k} |\mathcal{B}_{i_1,...,i_{k-2},k}^X|+(k-1)|\mathcal{B}_{1,...,k}^X|)
\end{array}
\]
Hence,
\begin{multline}
\eta \leq\\
(\sum\limits_{r=2}^{k} (r-1)\sum\limits_{1\leq i_1<...<i_r\leq k} \bar{\varepsilon}_{i_1,...,i_r}^X(\pi))-(\sum\limits_{r=2}^{k} (r-1)\sum\limits_{1\leq i_1<...<i_{r-1}< k} |\mathcal{B}_{i_1,...,i_{r-1},k}^X|)\\
\leq \sum\limits_{r=2}^{k-1} (r-1)\sum\limits_{1\leq i_1<...<i_r<k} |\mathcal{B}_{x_{i_1},...,x_{i_r}}^X|,
\end{multline}
where the last inequality holds because $\bar{\varepsilon}_{i_1,...,i_r}^X(\pi) \leq |\mathcal{B}_{x_{i_1},...,x_{i_r}}^X|$, for any $r\leq k$ and $1\leq i_1<...<i_r\leq k$.
\end{proof}
For $X=\{x_1,...,x_k\}$, let $\sigma$ be an arbitrary permutation on $\{1,..,k\}$ such that
\[
d_T^{(n)}(x_{\sigma(k)},X)=\min\limits_{i=1...k} d_T^{(n)}(x_i,X),
\]
Consider the relabelling $x_i^\sigma:=x_{\sigma(i)}$, for $i=1,...,k$, for elements of $X$.
So we can denote
\[
\mathcal{O}_n(X):=\sum\limits_{r=2}^{k-1} (r-1)\sum\limits_{1\leq i_1<...<i_r<k} |\mathcal{B}_{x_{i_1}^\sigma,...,x_{i_r}^\sigma}^X|.
\]
\begin{remark}\label{remark:upperb}
Of course for $X=\{x_1,...,x_k\}\subset S_n$ with maximum distance $d^{(n)}(x_i,x_j)=n-1$, for $i\neq j$, $\mathcal{O}_n(X)=0$ and therefore, as we already saw in Proposition \ref{extreme}, every median picks its adjacencies from the union of adjacencies of $k$ permutations. But the opposite is not true, namely, there exist sets of permutations $X$ such that $\mathcal{O}_n(X)=0$, but adjacency sets of permutations have intersections with each others. So Theorem \ref{thmalpha} gives a stronger statement when permutations of $X$ are not located at the maximum distance of each others but $\mathcal{O}_n(X)=0$ and in this case we still have the same property. For example, consider three permutations
\[
id=id^{(6)}= 1 \ 2 \ 3 \ 4 \ 5 \ 6 ,
\]
\[
x= 4 \ 6 \ 5 \ 1 \ 3 \ 2,
\]
and
\[
y=4 \ 2 \ 6 \ 5 \ 1 \ 3,
\]
and let $X=\{id,x,y\}$. We have $\mathcal{A}_{id,x}=\{\{2,3\},\{5,6\}\}$, $\mathcal{A}_{id,y}=\{\{5,6\}\}$, $\mathcal{A}_{x,y}=\{\{5,6\},\{1,5\},\{1,3\}\}$, and $\mathcal{A}_{id,x,y}=\{\{5,6\}\}$. Then $d_T^{(n)}(id,X)=7$, $d_T^{(n)}(x,X)=5$, and $d_T^{(n)}(y,X)=6$, and thus
\[
\mathcal{O}_n(X)=|\mathcal{A}_{id,y}\setminus \mathcal{A}_{id,x,y}|=0.
\]
\end{remark}

Motivated by Theorem \ref{thmalpha} and rewriting
\[
\LL_{n,0}(X):=\{\pi\in S_n: |\mathcal{A}_\pi\setminus \bigcup\limits_{x\in X} \mathcal{A}_x|\leq 0\},
\]
for $c\geq 0$ we define
\[
\LL_{n,c}(X):=\{\pi\in S_n: |\mathcal{A}_\pi\setminus \bigcup\limits_{x\in X} \mathcal{A}_x|\leq c\},
\]
and also let
\begin{multline}
\LL_{n,k,c}^{-1}(\pi)=\{(x_1,...,x_k)\in S_n^k: \pi \in \LL_{n,c}(\{x_1,...,x_k\})\}=\\
\{(x_1,...,x_k)\in S_n^k: |\mathcal{A}_\pi\setminus \bigcup\limits_{x\in X} \mathcal{A}_x|\leq c\}.
\end{multline}
Note that the left-invariance property of the breakpoint distance implies
\[
\A_{x,y}=\A_{\pi x,\pi y},
\]
for $\pi,x,y\in S_n$. Therefore, for any $\pi,x\in S_n$ and $X=\{x_1,...,x_k\}\subset S_n$, 
\[
d_T^{(n)}(x,X)=d_T^{(n)}(\pi x,\pi X),
\]
where $\pi X=\{\pi x_1,...,\pi x_k\}$. This yields
\[
\pi \mathcal{M}_n(x)=\mathcal{M}_n(\pi x),
\]
and therefore, for any $x,y\in S_n$.
\[
|\mathcal{M}_{n,k}^{-1}(x)|=|\mathcal{M}_{n,k}^{-1}(y)|.
\]
Also, denoting the bp median value of $X$ by $\mu_n(X)$, we can write
\[
\mu_n(X)=\mu_n(\pi X).
\]
On the other hand, for $\pi\in S_n$ and $k$-tuple $(x_1,...,x_k)\in S_n^k$, denote
\[
\pi(x_1,...,x_k)=(\pi x_1,...,\pi x_k).
\]
Similarly to the median inverse case, write
\[
\begin{array}{ll}
\mathscr{L}_{n,k,c}^{-1}(\pi x)&=\{(\pi x_1,...,\pi x_k): \ |\A_{\pi x}\setminus \cup_{i=1}^k \A_{\pi x_i}|\}\\
                               &=\{\pi(x_1,...,x_k): \ |\A_{x}\setminus \cup_{i=1}^k \A_{x_i}|\}\\
                               &=\pi \mathscr{L}_{n,k,c}^{-1}(x),
\end{array}
\]
and thus, for any $x,y\in S_n$
\[
|\mathscr{L}_{n,k,c}^{-1}(x)|=|\mathscr{L}_{n,k,c}^{-1}(y)|,
\]
since $\pi(x_1,...,x_k)\mapsto (\pi x_1,...,\pi x_k)$ is a bijection.

Let $\xi_1^{(n)},...,\xi_k^{(n)}$ be $k$ permutations chosen uniformly and independently at random from $S_n$. Intuitively, Proposition \ref{extreme} implies that for any sequence $(a_n)_{n\in\N}$, for which $a_n\rightarrow \infty$ and $a_n/n\rightarrow 0$, as $n\rightarrow \infty$, and any sequence of permutations $(\pi_n)_{n\in \N}$, with $\pi_n\in S_n$, the probability that 
\[
(\xi_1^{(n)},...,\xi_k^{(n)})\in \LL_{n,k,a_n}^{-1}(\pi_n),
\] 
somehow gives an upper bound for the probability that 
\[
(\xi_1^{(n)},...,\xi_k^{(n)})\in \mathcal{M}_{n,k}^{-1}(\pi_n).
\]
\begin{theorem}
Let $(a_n)_{n\in \N}$ be a sequence of real numbers diverging to $\infty$, such that $a_n/n\rightarrow 0$, as $n\rightarrow \infty$. Let $\xi_1^{(n)},...,\xi_k^{(n)}$ be $k$ permutations chosen uniformly and independently at random from $S_n$. Then for any arbitrary sequence of permutations $(\pi_n)_{n\in \N}$, with $\pi_n\in S_n$,
\begin{equation}\label{upperbound}
\limsup_n \p((\xi_1^{(n)},...,\xi_k^{(n)})\in \mathcal{M}_{n,k}^{-1}(\pi_n)) \leq \limsup_n \p((\xi_1^{(n)},...,\xi_k^{(n)})\in \LL_{n,k,a_n}^{-1}(\pi_n)).
\end{equation}
\begin{proof}
Let $X=\{x_1,...,x_k\}\subset S_n$, and consider sets $U_2,...,U_k\subset X$, such that for a fixed $i_1\neq i_2$ in $\{1,...,k\}$,
\[
U_2=\{x_{i_1},x_{i_2}\} \subsetneqq U_3 \subsetneqq ... \subsetneqq U_k=X,
\]
that is for any $l=2,...,k-1$, $|\mathcal{A}_{x_{l+1}}\setminus \mathcal{A}_{x_l}|=1$. Then by definition
\[
\bigcup_{i=1}^k \mathcal{B}_{U_i}^X=\mathcal{A}_{x_{i_1},x_{i_2}}.
\]
This yields
\[
|\mathcal{O}_n(X)|\leq \sum\limits_{i<j<k}|\mathcal{A}_{x_i,x_j}|.
\]
So if $|\mathcal{O}_n(X)|\geq c$, for $c\geq 0$, then for at least one pair of points in $X$, namely $x,y$,
\[
|\mathcal{A}_{x,y}|\geq \frac{c}{{k-1 \choose 2}}.
\]
Setting
\begin{small}
\[
\tau=\min\{i\in \{1,...,k\}: \ d_T^{(n)}(\xi_i^{(n)},\{\xi_1^{(n)},...,\xi_k^{(n)}\})\leq d_T^{(n)}(\xi_j^{(n)},\{\xi_1^{(n)},...,\xi_k^{(n)}\}), \ \ for \ j=1,...,k\}, 
\]
\end{small}
the last inequality implies that
\[
\begin{array}{l}
\p(\mathcal{O}_n(\{\xi_1^{(n)},...,\xi_k^{(n)}\})\geq a_n)\leq \\ \\
\p(\bigcup\limits_{{\scriptsize \begin{array}{c}i<j \\ i,j\neq \tau \end{array}}} \{|\mathcal{A}_{\xi_i^{(n)},\xi_j^{(n)}}|\geq \frac{a_n}{{k-1 \choose 2}}\})\leq \\
\sum\limits_{{\scriptsize \begin{array}{c}i<j \\ i,j\neq \tau \end{array}}} \p(|\mathcal{A}_{\xi_i^{(n)},\xi_j^{(n)}}|\geq \frac{a_n}{{k-1 \choose 2}})\rightarrow 0,
\end{array}
\]
as $n\rightarrow \infty$, by Proposition \ref{extreme}. Thus, (\ref{upperbound}) follows from
\[
\p((\xi_1^{(n)},...,\xi_k^{(n)})\in \mathcal{M}_{n,k}^{-1}(\pi_n)\setminus \LL_{n,k,a_n}^{-1}(\pi_n))\rightarrow 0,
\]
as $n\rightarrow \infty$, and
\[
\begin{array}{l}
\p(\xi_1^{(n)},...,\xi_k^{(n)})\in \mathcal{M}_{n,k}^{-1}(\pi_n))\leq \\
\p((\xi_1^{(n)},...,\xi_k^{(n)})\in \mathcal{M}_{n,k}^{-1}(\pi_n)\setminus \LL_{n,k,a_n}^{-1}(\pi_n))+ \p((\xi_1^{(n)},...,\xi_k^{(n)})\in \LL_{n,k,a_n}^{-1}(\pi_n)).
\end{array}
\]
\end{proof}
\end{theorem}
\begin{remark}
Note that the condition $a_n/n\rightarrow 0$, as $n\rightarrow 0$, is not necessary and the statement of the theorem remains true for a general diverging sequence $(a_n)_{n\geq 0}$. But in the absence of the mentioned condition the inequality will not be so helpful. In fact, the smaller the order of $a_n$, the better the upper bound.
\end{remark}
To find
\[
\p((\xi_1^{(n)},...,\xi_k^{(n)})\in \LL_{n,k,a_n}^{-1}(\pi_n))=\frac{|\LL_{n,k,a_n}^{-1}(\pi_n)|}{(n!)^k},
\]
we count the number of elements in $\LL_{n,k,c}^{-1}(\pi)$ in terms of $\mathcal{H}_\pi^{(n)}(J_i)$, for $c>0$, in the next theorem. First, for $c>0$, let
\[
\mathscr{P}_{k,c}^{(n)}(\pi)=\{(J_1,...,J_k)\in (\mathcal{I}^{(n)})^k: \ |\mathcal{A}_\pi\setminus \bigcup\limits_{i=1}^k J_i|\leq c\}.
\]
Letting $c=0$, this definition is identical with the one in (\ref{index-partition}).
Similarly to the case of $c=0$, let $\mathcal{J}=(J_1,...,J_k),\mathcal{J}'=(J_1',...,J_k')\in \mathscr{P}_{k,c}^{(n)}(\pi)$, such that $\mathcal{J}\neq \mathcal{J}'$, then
\[
(\HH_{\pi}^{(n)}(J_1)\times ...\times \HH_{\pi}^{(n)}(J_k)) \cap (\HH_{\pi}^{(n)}(J_1')\times ...\times \HH_{\pi}^{(n)}(J_k'))=\emptyset.
\]
Now, if $(x_1,...,x_k)\in \mathscr{L}_{n,k,c}^{-1}(\pi)$, then there exist at most $c$ adjacencies of $\pi$ that are not in $\cup_{i=1}^k \mathcal{A}_{x_i}$. Therefore, there exists $(J_1,...,J_k)\in \mathscr{P}_{k,c}^{(n)}(\pi)$ such that, for any $i=1,...,k$, $\mathcal{A}_{\pi,x_i}=J_i$ implying that $(x_1,...,x_k)\in \HH_{\pi}^{(n)}(J_1)\times ...\times \HH_{\pi}^{(n)}(J_k)$. On the other hand, if
\[
(x_1,...,x_k)\in \bigcup\limits_{(\tilde{J}_1,...,\tilde{J}_k)\in \mathscr{P}_{k,c}^{(n)}(\pi)} \HH_{\pi}^{(n)}(\tilde J_1)\times ...\times \HH_{\pi}^{(n)}(\tilde J_k),
\]
then there exists $(J_1,...,J_k)\in \mathscr{P}_{k,c}^{(n)}(\pi)$ such that $x_i\in \HH_{\pi}^{(n)}(J_i)$, for $i=1,...,k$, and so $\mathcal{A}_{\pi,x_i}=J_i$. Thus
\[
\begin{array}{l}
\mathcal{A}_\pi \setminus \bigcup\limits_{i=1}^k J_i=\mathcal{A}_\pi \setminus  \bigcup\limits_{i=1}^k \mathcal{A}_{\pi,x_i} =\mathcal{A}_\pi \setminus \bigcup\limits_{i=1}^k \mathcal{A}_{x_i}.
\end{array}
\]
Therefore, $|\mathcal{A}_\pi \setminus \bigcup\limits_{i=1}^k \mathcal{A}_{x_i}|\leq c$, and thus, $(x_1,...,x_k)\in \mathscr{L}_{n,k,c}^{-1}(\pi)$.
\begin{theorem}
Let $n,k$ be natural numbers, and let $c\geq 0$ be a real number. Also let $\pi$ be a permutation in $S_n$. Then
\[
\mathscr{L}_{n,k,c}^{-1}(\pi)=\bigcup\limits_{(\tilde{J}_1,...,\tilde{J}_k)\in \mathscr{P}_{k,c}^{(n)}(\pi)} \HH_{\pi}^{(n)}(\tilde J_1)\times ...\times \HH_{\pi}^{(n)}(\tilde J_k),
\]
and
\[
|\mathscr{L}_{n,k,c}^{-1}(\pi)|=\sum\limits_{(\tilde{J}_1,...,\tilde{J}_k)\in \mathscr{P}_{k,c}^{(n)}(\pi)} \prod\limits_{i=1}^k |\HH_{\pi}^{(n)}(\tilde J_i)|.
\]
\end{theorem}
As we saw, to estimate $\mathcal{M}_{n,k}^{-1}(\pi)$, we made use of $\mathscr{L}_{n,k,c}^{-1}(\pi)$, for $c\geq 0$, and to this end we need to count the number of elements of $\HH_{\pi}^{(n)}( J_i)$ for $(J_1,...,J_k)\in \mathscr{P}_{k,c}^{(n)}(\pi)$. In the rest of the chapter, we give a representation for the number of elements in $\HH_{\pi}^{(n)}( J_i)$.

For $J\in \mathcal{I}^{(n)}$, let
\[
\mathcal{R}_n(J)=\{\pi\in S_n: \ J\subset \mathcal{A}_\pi\}.
\]
Then, by inclusion-exclusion principle we have
\begin{equation}\label{inclusion-exclusion}
|\HH_{\pi}^{(n)}( I)|=\sum\limits_{I\subset J\subset\mathcal{A}_\pi} (-1)^{|J\setminus I|}|\mathcal{R}_n(J)|.
\end{equation}
To simplify this further, we introduce the type of a segment set $J$, $I\subset J\subset \mathcal{A}_\pi$, and we will see that the value of $|\mathcal{R}_n(.)|$ is identical for two segment sets of the same type. Formally, recall that $\bar{I}_\pi:=\mathcal{A}_\pi\setminus I$, and denote by $\bar{I}_\pi^{(i)}$, for $i=1,...,\|I\|+1$, the $i$-th segment of $\bar{I}_\pi$ ($i$-th from left when considered as a segment of $\pi$). Note that $\bar{I}_\pi^{(1)}$and $\bar{I}_\pi^{(\|I\|+1)}$ may be empty segments. 
The type of a segment set $J\in \mathcal{I}^{(n)}$, where $I\subset J\subset \mathcal{A}_\pi$, with respect to $\pi$ and $I$, is identified by
\[
\lambda:=(\lambda_1,...,\lambda_{\|I\|+1}),
\]
where, for $i=1,...,\|I\|+1$, $\lambda_i$ is identified by the quadruple
\[
\lambda_i:=(\lambda_i^{(1)},\lambda_i^{(2)},\lambda_i^{(3)},\lambda_i^{(4)})\in \N\times \N\times \{0,1\}\times \{0,1\},
\]
where
$\lambda_i^{(1)}:=|J\cap \bar{I}_\pi^{(i)}|$ is the number of adjacencies common to $J$ and $\bar{I}_\pi^{(i)}$; $\lambda_i^{(2)}:=\|J\cap \bar{I}_\pi^{(i)}\|$ is the number of segments of intersection of $J$ and $\bar{I}_\pi^{(i)}$; $\lambda_1^{(3)}=0$ and, for $i=2,...,\|I\|+1$, $\lambda_i^{(3)}=1$ if the most left adjacency of $\bar{I}_\pi^{(i)}$ is also in $J$ and otherwise $\lambda_i^{(3)}=0$; and finally $\lambda_{\|I\|+1}^{(4)}=0$, and for $i=1,...,\|I\|$, $\lambda_i^{(4)}=1$ if the most right adjacency of $\bar{I}_\pi^{(i)}$ is also in $J$ and otherwise $\lambda_i^{(4)}=0$. We need the following lemmas.
\begin{lemma}\label{permutation}
Let $I$ be a segment set of $S_n$ with $m$ adjacencies and $k$ segments. Then the number of permutations in $S_n$ containing $I$ is equal to  $2^{k}(n-m)!$.
\end{lemma}

\begin{proof}
As the segment set $I$ has $m$ adjacencies and $k$ segments, each permutation containing $I$ has $n-m-k$ isolated points with respect to $I$. Therefore, noting that segments have two directions, we have $2^{k}(k+(n-m-k))!$ permutations containing $I$.
\end{proof}
\begin{lemma}\label{volume of R}
Let $I,J\in \mathcal{I}^{(n)}$ and $\pi\in S_n$, such that $I\subset J\subset \mathcal{A}_\pi$. Let $\lambda=(\lambda_i)_{1\leq i\leq \|I\|+1}$, with $\lambda_i=(\lambda_i^{(1)},\lambda_i^{(2)},\lambda_i^{(3)},\lambda_i^{(4)})$, for $i=1,...\|I\|+1$ be the type of $J$ with respect to $\pi$ and $I$. Then
\[
|\mathcal{R}_n(J)|=2^{\{ |I|+\sum\limits_{i=1}^{\|I\|+1} \lambda_i^{(2)} - \sum\limits_{i=1}^{\|I\|+1} (\lambda_i^{(3)}+\lambda_i^{(4)})\}}(n-|I|-\sum\limits_{i=1}^{\|I\|+1} \lambda_i^{(1)})!.
\] 
\end{lemma}
\begin{proof}
We have
\[
\|J\|=\|I\|+\sum\limits_{i=1}^{\|I\|+1} \lambda_i^{(2)} - \sum\limits_{i=1}^{\|I\|+1} (\lambda_i^{(3)}+\lambda_i^{(4)}),
\]
and also the number of adjacencies of $J$ is equal to
\[
|J|=|I|+\sum\limits_{i=1}^{\|I\|+1} \lambda_i^{(1)}.
\]
Therefore, Lemma \ref{permutation} finishes the proof.
\end{proof}
\begin{lemma}\label{number of J-type given}
Let $\pi\in S_n$ and $I\in \mathcal{I}^{(n)}$. The number of segment sets $J$, with $I\subset J\subset \mathcal{A}_\pi$ and with type $\lambda=(\lambda_i)_{1\leq i\leq \|I\|+1}$ with respect to $\pi$ and $I$, where $\lambda_i=(\lambda_i^{(1)},\lambda_i^{(2)},\lambda_i^{(3)},\lambda_i^{(4)})$ for $i=1,...\|I\|+1$, is
\[
\prod\limits_{i=1}^{\|I\|+1} {\lambda_i^{(1)}-1 \choose \lambda_i^{(2)}-1}{|\bar{I}_\pi^{(i)}|-\lambda_i^{(1)}-1 \choose \lambda_i^{(2)} -\lambda_i^{(3)}-\lambda_i^{(4)}}.
\] 
\end{lemma}
\begin{proof}
The idea is to consider segment $\bar{I}_\pi^{(i)}$ as a permutation and count number of possible ways one can choose a segment set $\tilde J_i$ from it with $\lambda_i^{(1)}$ number of adjacencies and $\lambda_i^{(2)}$ number of segments. More explicitly, for $i=1,...,\|I\|+1$, if $(\lambda_i^{(3)},\lambda_i^{(4)})=(1,1)$, then the number of ways we can do this is equal to the number of solutions of two independent equations
\[
\mathbb{X}_1+...\mathbb{X}_{\lambda_i^{(2)}}=\lambda_i^{(1)},
\]
with $\mathbb{X}_i\geq 1$, for $i=1,...,\lambda_i^{(2)}$, and
\[
\mathbb{Y}_2+...+\mathbb{Y}_{\lambda_i^{(2)}}=|\bar{I}_\pi^{(i)}|-\lambda_i^{(1)},
\]
with $\mathbb{Y}_i\geq 1$, for $i=2,...,\lambda_i^{(2)}-1$, which is equal to
\[
{\lambda_i^{(1)}-1 \choose \lambda_i^{(2)}-1}{|\bar{I}_\pi^{(i)}|-\lambda_i^{(1)}-1 \choose \lambda_i^{(2)} -2}={\lambda_i^{(1)}-1 \choose \lambda_i^{(2)}-1}{|\bar{I}_\pi^{(i)}|-\lambda_i^{(1)}-1 \choose \lambda_i^{(2)} -\lambda_i^{(3)}-\lambda_i^{(4)}}
\]
since $\lambda_i^{(3)}+\lambda_i^{(4)}=2$. Similarly, for the cases $(\lambda_i^{(3)},\lambda_i^{(4)})=(0,0),(0,1),(1,0)$, we can prove that the number of ways one can choose a segment set $\tilde J_i$ from it with $\lambda_i^{(1)}$ number of adjacencies and $\lambda_i^{(2)}$ number of segments is
\[
{\lambda_i^{(1)}-1 \choose \lambda_i^{(2)}-1}{|\bar{I}_\pi^{(i)}|-\lambda_i^{(1)}-1 \choose \lambda_i^{(2)} -\lambda_i^{(3)}-\lambda_i^{(4)}}.
\]
Multiplying all possibilities for $i=1,...,\|I\|+1$ yields the result.
\end{proof}
We are ready to count the number of elements of $\HH_\pi^{(n)}(I)$. 
\begin{theorem} \label{seth}
Let $\pi$ be a permutation in $S_n$, and $I\in \mathcal{I}^{(n)}$ be a segment set contained in $\pi$. Then
\begin{multline}
|\HH_{\pi}^{(n)}(I)|=\sum\limits_\lambda \left\lbrace (-1)^{\sum\limits_{i=1}^{\|I\|+1} \lambda_i^{(1)}} \prod\limits_{i=1}^{\|I\|+1} \left[  {\lambda_i^{(1)}-1 \choose \lambda_i^{(2)}-1}{|\bar{I}_\pi^{(i)}|-\lambda_i^{(1)}-1 \choose \lambda_i^{(2)} -\lambda_i^{(3)}-\lambda_i^{(4)}} \right]  \times \right. \\
\left. 2^{\{ |I|+\sum\limits_{i=1}^{\|I\|+1} \lambda_i^{(2)} - \sum\limits_{i=1}^{\|I\|+1} (\lambda_i^{(3)}+\lambda_i^{(4)})\}}(n-|I|-\sum\limits_{i=1}^{\|I\|+1} \lambda_i^{(1)})! \right\rbrace,
\end{multline}
where the summation is over  all $\lambda=((\lambda_i^{(1)},\lambda_i^{(2)},\lambda_i^{(3)},\lambda_i^{(4)}))_{1\leq i\leq \|I\|+1}$ with 
\[
(\lambda_i^{(1)},\lambda_i^{(2)},\lambda_i^{(3)},\lambda_i^{(4)})\in \{1,...,|\bar{I}_\pi^{(i)}|\} \times \{1,...,\min\{\lambda_i^{(1)}, |\bar{I}_\pi^{(i)}|+1-\lambda_i^{(1)}\}\}\times \{0,1\}\times \{0,1\}.
\]
\end{theorem}
\begin{proof}
The proof is a direct application of Lemmas \ref{volume of R} and \ref{number of J-type given}, and (\ref{inclusion-exclusion}), inclusion-exclusion principle.
\end{proof}




\begin{thebibliography}{99}


\bibitem{bryant98} Bryant, D., \emph{The complexity of the breakpoint median problem}. Centre de recherches mathematiques, (1998).


\bibitem{caprara03} Caprara, A., \emph{The reversal median problem}, INFORMS Journal on Computing,15 (2003), pp. 93--113.
 
 
\bibitem{fertincombinatorics} Fertin, G. and Labarre, A. and Rusu, I. and Tannier, E. and Vialette, S., \emph{Combinatorics of genome rearrangements}, The MIT Press, 2009.

\bibitem{jam14} Jamshidpey, A. and Jamshidpey, A. and Sankoff, D., \emph{Sets of medians in the non-geodesic pseudometric space of unsigned genomes with breakpoints}, BMC genomics, 15 (2014), p. S3.

\bibitem{sankoff97} Sankoff, D. and Blanchette, M., \emph{The median problem for breakpoints in comparative genomics}, Computing and combinatorics, (1997), pp. 251--263.

\bibitem{tannier09} Tannier, E. and Zheng, C. and Sankoff, D., \emph{Multichromosomal median and halving problems under different genomic distances}, BMC bioinformatics, 10 (2009), p. 120.

\end{thebibliography}

\end{document}